\documentclass[11 pt]{amsart}

\usepackage{amssymb,amsmath,latexsym,amsthm}
\usepackage[english]{babel}

\newtheorem{thm}{Theorem}[section]
\newtheorem{prop}{Proposition}[section]
\newtheorem{lemma}{Lemma}[section]
\newtheorem{cor}{Corollary}[section]

\numberwithin{equation}{section}

\begin{document}

\title[Integral points on a certain family of elliptic curves]{Integral points on a certain family of elliptic curves}

\author{\sc Shabnam Akhtari }
\address{Shabnam Akhtari\\
Fenton Hall\\
University of Oregon\\
Eugene, OR 97403-1222 USA}
\email{akhtari@uoregon.edu} 
\urladdr{http://pages.uoregon.edu/akhtari}



\thanks{I am extremely grateful to  the anonymous referee who  read the earlier version of this manuscript  carefully and patiently and provided plenty of insightful comments and suggestions.}

\maketitle

\begin{abstract}
The Thue-Siegel method is used to obtain an upper bound for the number of primitive integral solutions to a family of quartic Thue's inequalities. This will provide an upper bound for the number of integer points on  a family of elliptic curves with j-invariant equal to $1728$.
\end{abstract}

\section{Introduction and  statements of the main results}\label{Intro}

A well-known theorem  of Siegel, in its simplest form, is the fact
that a nonsingular elliptic curve contains only finitely many integer  points.
 Let
 $$
E: y^2 = x^3 + A x + B
$$
and 
$$ H = \max \{ |A| , |B|\}.
$$
 Schmidt  in \cite{Sch} established some upper bounds for the number and the size  of integer points on plane curves of genus $1$. His work implies, for given $\epsilon > 0$,  the upper bound
 $$
 c(\epsilon) H^{2 +\epsilon}
 $$
 upon the number of integral points on $y^2 = x^3 + A x + B$, where   the constant $c(\epsilon)$ is effectively computable.
  Schmidt conjectured that for every $\epsilon > 0$, there exists a constant $c(\epsilon)$ such that the number of integral points on $y^2 = x^3 + A x + B$ is bounded above by
 $c(\epsilon) H^{\epsilon}$.
 Evertse and Silverman in \cite{ES} gave an upper bound for the number of integer points on elliptic curves
 \begin{equation}\label{ESthm}
 y^2 = f(x) =x^3 + bx^2 + cx +d,
 \end{equation}
 which depends on the class number of the splitting field of $x^3 + bx^2 + cx +d$. They showed that for any $\epsilon > 0$, the number of integer solution to an  equation of the form \eqref{ESthm}, with $b, c, d \in \mathbb{Z}$ and non-zero discriminant $\Delta(f)$, is bounded by $c(\epsilon) |\Delta(f)|^{1/2 + \epsilon}$ for some effectively computable $c(\epsilon) > 0$.   
 Also Helfgott and Venkatesh  in \cite{HV} provided some nice improved upper bounds on the number of integer points (and $S$-integer points) on elliptic curves.   Their work implies  an upper bound in terms of the rank of the elliptic curves.  Moreover, they showed, among other things,  that the number of integral points on any given elliptic curve $E$ is  $\ll |\Delta|^{\theta + \epsilon}$ for every sufficiently small $\epsilon$, where $\Delta$ is the discriminant of $E$ and $\theta = 0.20070\ldots$.
 
 In this manuscript we will consider a certain family of elliptic curves with  $j$-invariant $1728$ and use different tools to provide an upper bound for the number of integer points on such curves. We will improve upon some results  obtained  by Walsh in \cite{Wal}. The results  presented here are substantially more restrictive than those in  \cite{ES},  \cite{HV}  and \cite{Sch}. However our results and techniques   are quite different and improve upon the previous results in certain cases.

Elliptic curves with $j$-invariant $1728$ are one of the important families in the arithmetic theory of elliptic curves.  Every such curve has an equation of the form $Y^2 = X^3 \pm N X$, where $N$ is a $4$th-power-free positive integer (see \cite{Sil}, for example). Here we will consider  the problem of counting the number of   integral points on curves of this form, though our focus is mainly on the cubic equation
\begin{equation}\label{maincequ}
Y^2 = X^3 - N X,
\end{equation}
where  $N$ is a square-free 
positive integer. 
In Section \ref{+N} we  briefly discuss the problem of counting the integral points on the curve $Y^2 = X^3 +N X$ and explain  how this problem is different from our main problem. 
It turns out that the integral points on the equation \eqref{maincequ} can be seen as integral points  on  a finite number of curves given by quartic equations of the form
\begin{equation}\label{mainqequ}
X^2 - D Y^4 = k,
\end{equation}
where  $D >1$ is a square-free integer  and $k < 0$  is a negative  integer relatively prime to $D$. 
To study the integral solutions of \eqref{mainqequ}, we can look at the integral solutions of the quadratic equation 
\begin{equation*}
\mathfrak{X}^2 - D \mathfrak{Y}^2 = k
\end{equation*}
and detect those with $\mathfrak{Y}$ a perfect square. 
For the positive non-square  integer $D$, let
$$
\epsilon_{D} = T + U\sqrt{D},
$$
with $T, U$ positive integers, denote the minimal unit greater than $1$ in the ring $\mathbb{Z}[\sqrt{D}]$. Notice that  $\mathbb{Z}[\sqrt{D}]$ is an order in the ring of integers of the number field $\mathbb{Q}(\sqrt{D})$, and since we are looking for integer points on various curves, we prefer to work with $\mathbb{Z}[\sqrt{D}]$ instead of the ring of integers of $\mathbb{Q}(\sqrt{D})$.

 \begin{thm}\label{main=}
Let $k$ be a negative integer. The quartic equation \eqref{mainqequ} has at most $$384 . \, 2^{\omega(k)} \epsilon_{D}^{3/2} \sqrt{|k|/ 2D}$$ solutions in positive  integers $X$ and $Y$,  where $\omega(k)$ denotes the number of prime factors of the integer $k$.
\end{thm}

\begin{thm}\label{mainqe}
Let $k$ be a negative integer satisfying
$$
|k| \geq  
\frac{\pi}{2^{14} 3^{11/2}} \frac{\epsilon_{D}^{12} }{D^{13/2}}.
$$
The quartic equation \eqref{mainqequ} has at most $$40 . 2^{\omega(k)}$$ solutions in positive  integers $X$ and $Y$, where $\omega(k)$ denotes the number of prime factors of the integer $k$.
\end{thm}

The unit  $\epsilon_{D}$ is bounded from above.  In fact, we have 
\begin{equation}\label{Lenepsilon}
\epsilon_{D} < \exp\left(D^{1/2} \left(\log (4D) +2 \right)\right)
\end{equation}
(see \cite{Len}).
Therefore, in the statement of Theorem \ref{mainqe}, one can replace the given lower  bound on $|k|$ with an explicit function of $D$:
$$
|k| \geq  
\frac{\pi}{2^{14} 3^{11/2}} \frac{\exp\left(12 \, D^{1/2} \left(\log (4D) +2 \right)\right)
  }{D^{13/2}}.
$$

\begin{thm}\label{missproof}
Let $N$ be a positive square-free integer. The equation \eqref{maincequ} has at most $$384 \sqrt{N/2}  \sum_{D| N}  \frac{2^{\omega(N/D)}   \epsilon_{D}^{3/2} }{D}$$
solutions in integers $X$, $Y$. 
\end{thm}

The above theorems improve the results in   \cite{Wal}. Walsh in \cite{Wal} proved that there are at most $48 .2^{\omega(k)}$ integer solutions to \eqref{mainqequ} with
$$
|Y| > \frac{2^{5/4} |k|^{39/4} \epsilon_{D}^{45/4}}{D^{13/4}}.
$$
 Then he concluded that, if $N$ is a positive square-free integer,  there are at most $48 \sum_{D\mid N} 2^{\omega(D)}$ integer solutions to \eqref{maincequ} with 
$$
|X| > \max_{D \mid N, D > 1} \frac{ 2^{5/2} |N/D|^{39/2} \epsilon_{D}^{45/2}}{D^{11/2} }.
$$

In order to count the number of  integral points on the above cubic and quartic curves, we will reduce them to a family of quartic Thue equations $F(x , y) = m$, where $m \in \mathbb{Z}$. Let  
$$
F(x , y) = a_{0}x^{4} + a_{1}x^{3}y + a_{2}x^{2}y^{2} + a_{3}xy^{3} + a_{4}y^{4} 
$$
be a binary quartic form with integer coefficients. The discriminant $\Delta$ of $F$ is given by
$$
\Delta = a_{0}^{6} (\alpha_{1} - \alpha_{2})^{2}  (\alpha_{1} - \alpha_{3})^{2}   (\alpha_{1} - \alpha_{4})^{2}   (\alpha_{2} - \alpha_{3})^{2}  (\alpha_{2} - \alpha_{4})^{2}  (\alpha_{3} - \alpha_{4})^{2},
$$
where $\alpha_{1}$ , $\alpha_{2}$, $\alpha_{3}$ and $\alpha_{4}$ are  roots of the polynomial
$$
F(x , 1) =  a_{0}x^{4} + a_{1}x^{3} + a_{2}x^{2} + a_{3}x + a_{4} . 
$$
The ring of  invariants of $F$ is generated by two invariants  
$$
I = I_{F} = a_{2}^{2} - 3a_{1}a_{3} + 12a_{0}a_{4},
$$
and
$$
J =  J_{F} = 2a_{2}^{3} - 9a_{1}a_{2}a_{3} + 27 a_{1}^{2}a_{4} - 72 a_{0}a_{2}a_{4} + 27a_{0}a_{3}^{2},
$$
of weights $4$ and $6$, respectively.  Every invariant of $F$ is a polynomial in $I$ and $J$. In particular, for the discriminant $\Delta$, we have
$$
27\Delta = 4I^3 - J^2.
$$
In this manuscript we will  consider the forms $F$ for which the invariant $J=0$, so that  we have
$$
27\Delta = 4I^{3}.
$$
We will show
\begin{thm}\label{mainq<}
  Let  $F(x , y)$ be an irreducible binary quartic form with $I > 0$ and $J = 0$. Suppose that all four roots of $F(X , 1)$ are real and  $h$ is an integer  satisfying $h = \frac{\sqrt{3}\,   I^{1/2 - \epsilon}}{\pi}$, with  $0 < \epsilon < \frac{1}{2}$. Then the Thue inequality 
  \begin{equation}\label{mainineq}
\left|F(x , y)\right| \leq h
\end{equation}
  has at most $$
4\left[ \frac{\log \left( \frac{1}{2 \epsilon} - \frac{1}{2} \right)}{\log 3} \right] + 16
$$
solutions in  coprime integers $x$ and $y$ with $y \neq 0$, where $[.]$ denotes the greatest integer function and  $(x , y)$ and $(-x , -y)$ are counted as one solution.
\end{thm}

The upper  bound given in  the statement  of Theorem \ref{mainq<} depends on $\epsilon$ and therefore, on the integer $h$.  Setting $\epsilon > \frac{1}{4}$, for example, in Theorem \ref{mainq<} yields the following.
\begin{cor}
  Let  $F(x , y)$ be an irreducible binary quartic form with $I > 0$ and  $J = 0$. Suppose that all roots of polynomial $F(x , 1)$ are real and  $h$ is an integer satisfying $h < \frac{\sqrt{3}\,   I^{5/14}}{\pi}$. Then the inequality
  \begin{equation*}
\left|F(x , y)\right| \leq h
\end{equation*}
  has at most $ 16$ solutions in  coprime integers $x$ and $y$, with $y \neq 0$. Here $(x , y)$ and $(-x , -y)$ are counted as one solution.
\end{cor}

Let $G(x , y) \in \mathbb{Z}[x , y]$ be a form of degree $n$  that is irreducible over $\mathbb{Q}$ and  let $h \in \mathbb{Z}$. Bombieri and Schmidt \cite{Bom} showed that 
the number of solutions of $G(x , y) = h$ in co-prime integers $x$ and $y$ is at most 
$$
C_{2} \,  n^{1 + \omega(h)},
$$
where $C_{2}$ is an absolute constant. 
In \cite{AkhTA}  the author obtained an upper bound for the number of integer solutions to the Thue inequality $|G(x , y)| \leq h$, where $G$ is a binary form of degree $n \geq 3$ and with non-zero discriminant $\Delta$, and $h$ is an integer  smaller than $|\Delta|^{\frac{1}{4(n-1)}}$. One may apply that upper bound to the quartic Thue inequalities appearing in this manuscript. Using  properties specific  to this family, we  obtain better upper bounds here.
 In particular, the 
hypergeometric method of Thue and Siegel will be applied to the specific family of inequalities we are dealing with here. The method of Thue-Siegel cannot be used in more general cases. We refer the reader to \cite{Chu} for an overall study of the Thue-Siegel method. In \cite{Akh4}, we showed that  the Thue-Siegel method may be applied  to a quartic Thue equation $F(x , y) = h$, only if $J_{F} = 0$.

\section{Quartic Thue inequalities and 
Proof of Theorem \ref{mainq<}
}\label{QT}

In this section we will consider  Thue inequalities of the shape
$$ 
|F(x , y)| \leq h,
$$
where $F(x , y)$ is a quartic binary form. We will provide an upper bound for the number of co-prime integer  solutions (or \emph{primitive} solutions)  to such inequalities under the assumptions in Theorem \ref{mainq<}.

We  call  binary forms $F_{1}$ and $F_{2}$ equivalent if they are equivalent under $GL_{2}(\mathbb{Z})$-action (i.e. if there exist integers $b$, $c$, $d$ and $e$ such that
$$
F_{1}(bx + c y , dx + ey) = F_{2}(x , y)
$$
for  all $x$ and $y$, where $be - cd = \pm 1$). Denote by $N_{F}$ the number of solutions in integers $x$ and $y$ of the  equation
\begin{equation*}
\left|F(x , y)\right| = h.
\end{equation*}
 Note that if $F_{1}$ and $F_{2}$ are equivalent, then $N_{F_{1}} = N_{F_{2}}$, $I_{F_{1}} = I_{F_{2}}$ and $J_{F_{1}} = J_{F_{2}}$, where $I$ and $J$ are the invariants  defined in the Introduction.

 Let    
 $$
  H(x , y) = \frac{d^{2}F}{dx^{2}}  \frac{d^{2}F}{dy^{2}} - \left(\frac{d^{2}F}{dx dy}\right)^{2}
  $$
   denote the  Hessian associated to the quartic form $F(x , y)$.
  Then
 $$
 H(x , y)= A_{0}x^{4} + A_{1}x^{3}y + A_{2}x^{2}y^{2} + A_{3}xy^{3} + A_{4}y^{4},
 $$
  with
  \begin{eqnarray}\label{Ai2}\nonumber
 A_{0} & = & 3 (8a_{0}a_{2} - 3a_{1}^{2}) ,\\  \nonumber
 A_{1} & = & 12(6a_{0}a_{3} - a_{1}a_{2}), \\ 
 A_{2} & = &  6(3a_{1}a_{3} + 24a_{0}a_{4} -2a_{2}^{2}), \\  \nonumber
 A_{3} & = & 12(6a_{1}a_{4} - a_{2}a_{3}),\\ \nonumber
 A_{4} & = & 3(8a_{2}a_{4} - 3 a_{3}^{2}).
  \end{eqnarray}

 \begin{lemma}\label{P682}
 Suppose $F(x , y)$ is a quartic form with invariants $I$ and $J$ and Hessian $H(x , y)$. Let $\phi$ be a root of $X^3 -3I  X + J$. Then 
$$
-\frac{1}{9}H(x , y) + \frac{4}{3} \phi F(x , y) = \mathbf{m}(x , y)^2,
$$
where $\mathbf{m}(x , y)$ is a quadratic covariant of $F(x , y)$.
\end{lemma}
\begin{proof}
See part (vi) of Proposition $6$ of \cite{Cre}.
\end{proof}

\begin{lemma}\label{P882}
Let $F(x , y)$ be a quartic form with real coefficients and  leading coefficient $a_{0}$. Suppose that $F(Z , 1) = 0$ has $4$ real roots. Let $\phi_{1}$, $\phi_{2}$ and $\phi_{3}$ be the roots  of $X^3 -3I  X+ J$, with $4a_{0} \phi_{1} > 4a_{0} \phi_{2} > 4a_{0} \phi_{3}$. Set $\phi = \phi_{2}$. Then $\mathbf{m}(x , y)$ is a positive definite quadratic form with real coefficients, where $\mathbf{m}(x , y)$ is the covariant of $F(x , y)$ defined in Lemma $\ref{P682}$.
\end{lemma} 
\begin{proof}
This is part (ii) of Proposition $8$ of \cite{Cre}. 
Also see Lemma 3.4 of \cite{Akh4}.
\end{proof}

  A real quadratic form $f(x , y) = ax^{2} + bxy + cy^{2}$ is called \emph{reduced} if 
 $$
 |b| \leq a \leq c .
 $$
 It is well-known that every positive definite quadratic form is equivalent to a reduced one.
 Following Definition 4 of \cite{Cre},  we say that the quartic form $F(x , y) = a_{0}x^{4} + a_{1}x^{3}y + a_{2}x^{2}y^{2} + a_{3}x y^{3}+ a_{4}y^{4}$ with positive discriminant is \emph{reduced} if and only if the positive definite quadratic form  $\mathbf{m}(x ,y)$ is reduced.

  Suppose that the quartic form $F(x , y)$  is reduced, $J_{F} = 0$ and $I_{F} > 0$. Then in Lemma  \ref{P882}, we have $\phi_{1},  \phi_{3} \in \{ \sqrt{3I_{F}}, -\sqrt{3I_{F}}\}$ and $\phi_{2} = 0$.
  Taking $\phi = 0$ in Lemma \ref{P682}, we know that the algebraic covariant $\frac{-1}{9} H(x , y)$
is the square of a quadratic form, say
$$
\frac{-1}{9} H(x , y) = \mathbf{m}^2(x , y).
$$
 Using the fact that  $\mathbf{m}(Z) = \mathbf{m}(Z , 1)$ assumes a minimum equal to $ \frac{4AC - B^2}{4A} $ at $Z = \frac{-B}{2A}$, we showed in \cite{Akh4} that when $\theta_{1} , \theta_{2} \in \mathbb{R}$ and $\theta_{2} \neq 0$, we have 
\begin{equation}\label{mtheta}
\mathbf{m}(\theta_{1} , \theta_{2})  \geq 2\sqrt{I}\theta_{2}^2.
\end{equation}

  The following results from \cite{Akh4} provide an upper bound for the number of  ``large'' primitive solutions to quartic Thue inequalities of the shape $|F(x , y)| \leq h$. After stating these results, we will obtain an upper bound for the number of ``small'' primitive solutions, when $h < \frac{\sqrt{3}I}{\pi}$.  We call a primitive solution $(x , y)$ small if $|H(x , y)| < 4h^3\sqrt{\left|3I A_{4}\right|}$, and large if $|H(x , y)| \geq  4h^3\sqrt{\left|3I A_{4}\right|}$.

\begin{lemma}\label{red2}
Let $F(x , y) \in \mathbb{Z}[x , y]$ be a  quartic form  with $J = 0$. If $F$ is reduced then for every  $(x_{1}, y_{1}) \in \mathbb{Z}^2$ we have
 $$
 \left|H(x_{1} , y_{1})\right| \geq 36 \, I y_{1}^4,
 $$
 where $H(x , y)$ is the Hessian of $F(x , y)$.  \end{lemma}
\begin{proof}
This is Lemma 3.5 of \cite{Akh4}.
\end{proof}

\begin{prop}\label{rF2}
Let $F(x , y) \in \mathbb{Z}[x , y]$ be a  quartic form  with $J = 0$ that is irreducible over $\mathbb{Q}$.  Then 
$$
F(x , y) = \frac{1}{8\sqrt{3I |A_{4}|}} \left(\xi^{4}(x , y) - \eta^{4}(x , y) \right),
$$
where $\xi(x , y)$ and $\eta(x , y)$ are linear forms in $x$ and $y$, with
     $$
    \xi^4(x , y), \eta^4(x , y)  \in \mathbb{Q} \left(\sqrt{|A_{0}| I/3} \right)[x , y].   
   $$
   Moreover, if all roots of the  polynomial $F(Z , 1)$ are real then $I > 0$, $A_{0} < 0$, and $\xi(x , y)$ and $\eta(x , y)$ are complex conjugate linear forms.
 \end{prop}
\begin{proof}
This is Lemma 5.1 of \cite{Akh4}. 
\end{proof}

  The linear forms $\xi(x , y)$ and $\eta(x , y)$ in Proposition \ref{rF2} are called a pair of \emph{resolvent forms}. In what follows we will often write $\xi$ and $\eta$ for $\xi(x , y)$ and $\eta(x , y)$ at a particular point $(x , y)$ when that point is understood.
  Note that if $(\xi , \eta)$ is one pair of resolvent forms then there are precisely three others, given by $(i \xi , -i \eta)$, 
$(-\xi , -\eta)$ and $(-i\xi , i\eta)$, where $i = \sqrt{-1}$. 
For the  fixed pair of resolvent forms   $(\xi , \eta)$, we have  equation (22) of \cite{Akh4}:
 \begin{equation}\label{c62}
|\xi(x , y) \eta(x , y)| =  \frac{\left(H(x , y)^2|A_{4}|\right)^{1/4}}{\sqrt{3}}.
\end{equation}
  Also as a direct consequence of the definition of $\xi$ and $\eta$, it is shown in \cite{Akh4} (equation (25)) that if $x_{1}, y_{1}, x_{2}, y_{2} \in \mathbb{Z}$, with $x_{1}y _{2}- x_{2}y_{1} \neq 0$, then
  \begin{equation}\label{lb2}
 |\xi(x_{1}, y_{1}) \eta(x_{2}, y_{2}) - \xi(x_{2}, y_{2}) \eta(x_{1}, y_{1}) |   \geq 2\sqrt{I} \left| A_{4} \right|^{1/4}.
 \end{equation}

From now on, we fix a pair of resolvent forms.
Let $\omega$ be a fourth root of unity. 
We say that the integer pair $(x , y)$ is \emph{related} to $\omega$ if 
$$
\left|\omega - \frac{\eta(x , y)}{\xi(x , y)}\right| = \min_{0 \leq k \leq 3}\left|e^{2k\pi i/4} - \frac{\eta(x , y)}{\xi(x , y)}\right|. 
$$
 Let us define  $z = 1 - \left(\frac{\eta(x , y)}{\xi(x , y)}\right)^{4}$, where $(\xi , \eta)$ is a fixed pair of resolvent forms (in other words, $\frac{\eta}{\xi}$ is a fourth root of the complex number $1-z$).
We have
$$
 |1 - z| = 1  \    \   ,  \    \     |z| < 2 .
 $$
Note that when $F(x , y)$ is irreducible then $|z| \neq 2$. Because if  $|z| = 2$ then $\eta^4 = -\xi^4$, so $F(x , y) = \frac{1}{4 \sqrt{3IA_{4}}} \xi^4 (x , y)$.   This implies that $F(x , y)$ has one root with multiplicity $4$ and therefore the discriminant of $F(x , y)$ is $0$.
This contradicts the assumption  that  $F(x , y)$ is irreducible over $\mathbb{Q}$. 

 \begin{lemma}\label{6.12}
 Let  $\omega$ be a fourth root of unity and  $(x , y) \in \mathbb{Z}^2$ satisfy $F(x , y) =  \frac{1}{8\sqrt{3IA_{4}}}(\xi^{4}(x , y) - \eta^{4}(x , y)) =1$, with
 $$
\left|\omega - \frac{\eta(x , y)}{\xi(x , y)}\right| = \min_{0 \leq k \leq 3}\left|e^{2k\pi i/4} - \frac{\eta(x , y)}{\xi(x , y)}\right|. 
$$
Set $z = 1 - \left(\frac{\eta(x , y)}{\xi(x , y)}\right)^{4}$.
 If $|z| \geq 1$ then 
 \begin{equation}\label{Gap12}
\left|\omega  - \frac{\eta(x,y)}{\xi(x,y)}\right| \leq \frac{\pi}{8} |z|.
\end{equation}
If $|z| < 1$ then 
\begin{equation}\label{Gap22}
\left|\omega - \frac{\eta(x,y)}{\xi(x,y)}\right| < \frac{\pi}{12} |z|.
\end{equation}
  \end{lemma}
  \begin{proof}
  This is Lemma 6.1 of \cite{Akh4}.
  \end{proof}

\begin{lemma}\label{largexi}
Let $F (x , y) \in \mathbb{Z}[x , y]$ be a reduced quartic form satisfying the conditions in Proposition \ref{rF2}.  If all roots of the polynomial $F(Z , 1)$ are real then the inequality $|F(x , y)| \leq h$
has at most $12$ primitive solutions $(x , y)$ with 
$$
\left| \xi(x , y) \right|^4  \geq 4h^3\sqrt{\left|3I A_{4}\right|}.
$$
\end{lemma}
\begin{proof}
This was shown  in \cite{Akh4}  (see, in particular,  equation (29) through equation (30)).
\end{proof}

Note that in the above lemma, and generally in \cite{Akh4}, no restriction is assumed on the value of $h$. Instead, the upper bounds are given for the number of those  solutions that are large in terms of $h$.

Using a standard gap argument,  we will establish an upper  bound for the number of solutions 
$(x , y)$ to  \eqref{mainineq} that satisfy
$$
\left| \xi(x , y) \right|^4  < 4h^3\sqrt{\left|3I A_{4}\right|},
$$
when 
 $$
h \leq \frac{\sqrt{3}\,   I^{1/2 - \epsilon}}{\pi}, \,  \textrm{with}\,  0 < \epsilon < \frac{1}{2}.
$$

Suppose there are distinct solutions to $\left| F(x , y)\right| \leq  h$ indexed by $i$, say 
$(x_{i}, y_{i})$, that are related to a fixed fourth root of unity $\omega$ with $$
|\xi(x_{i+1} , y_{i+1})| \geq |\xi(x_{i} , y_{i})|.$$
Let
$$
F(x_{i} , y_{i}) = h_{i}, \,  F(x_{i+1} , y_{i+1}) = h_{i+1}, 
$$
with $|h_{i}|, |h_{i+1}| \leq h$. 
Set
 $$
 \eta_{i} = \eta(x_{i} , y_{i})\, \quad  \textrm{and} \, \quad \xi_{i} = \xi(x_{i} , y_{i}).
 $$
By Lemma \ref{red2} and \eqref{c62}, for every index $i$, we have
\begin{equation}\label{naturally}
|\xi_{i}| \geq  12^{1/4} I^{1/4} |A_{4}|^{1/8}, 
\end{equation}
provided that $y_{i} \neq 0$.  Using the definition of $z_{i}$ above, then  Proposition \ref{rF2} and the assumption that  $|h_{i}| \leq h < \frac{ \sqrt{3} I^{1/2 }}{\pi}$, we get
\begin{eqnarray*}
|z_{i}| &=& \left|1 - \frac{\eta^4(x , y)}{\xi^4(x , y)} \right|\\
& = & \left| \frac{8h_{i} \sqrt{ \left|3I\, A_{4}\right|}}{\xi_{i}^{4}} \right|\\
&\leq & \frac{3 I \sqrt{|A_{4}|}\,  8}{\pi |\xi^4(x , y)|}\\
&\leq & \frac{3 I \sqrt{|A_{4}|}\,  8}{\pi 12 I \sqrt{|A_{4}|}}\\
& = & 2/\pi < 1.
\end{eqnarray*}
We have
 \begin{eqnarray*}
& & |\xi_{i} \eta_{i+1} - \xi_{i+1} \eta_{i}| \\ \nonumber 
& = &\left|\xi_{i}( \eta_{i+1} - \omega \xi_{i+1}) - \xi_{i+1}( \eta_{i} - \omega \xi_{i})\right| \\ \nonumber 
& \leq & \left| \xi_{i}\xi_{i+1} \left(\frac{\eta_{i+1}}{\xi_{i+1}} -\omega \right)\right| + \left| \xi_{i}\xi_{i+1} \left(\frac{\eta_{i}}{\xi_{i}} -\omega \right)\right|\, (\textrm{by the triangle inequality})\\ \nonumber
&\leq & \frac{\pi}{12}\left(| \xi_{i}\xi_{i+1}z_{i+1}| + | \xi_{i}\xi_{i+1}z_{i}|\right) \,  (\textrm{from} \, (\ref{Gap22}))\\ \nonumber
& = & \frac{\pi}{12}\left(| \xi_{i}\xi_{i+1} \frac{\eta_{i+1}^4 - \xi_{i+1}^4}{\xi_{i+1}^4}| + | \xi_{i}\xi_{i+1}\frac{\eta_{i}^4 - \xi_{i}^4}{\xi_{i}^4}|\right) \\ \nonumber
&\leq&\frac{2 \pi}{3} h\sqrt{\left|3I\, A_{4}\right|}\left(\frac{|\xi_{i}|}{|\xi_{i+1}^{3}|} + \frac{|\xi_{i+1}|}{|\xi_{i}^{3}|}\right) \, (\textrm{by Proposition \ref{rF2}}).
\end{eqnarray*}
Since we have  assumed $|\xi_{i}| \leq |\xi_{i+1}|$, we get
$$
|\xi_{i} \eta_{i+1} - \xi_{i+1} \eta_{i}| \leq \frac{4 \pi}{3}  h\sqrt{\left|3I\, A_{4}\right|} \left( \frac{|\xi_{i+1}|}{|\xi_{i}^{3}|}\right).
$$
Combining this with (\ref{lb2}), we conclude 
\begin{equation}\label{GapN}
|\xi_{i+1}| \geq \frac{\sqrt{3}}{2 \pi h\, \left|A_{4}\right|^{1/4}}|\xi_{i}|^{3}.
\end{equation}
  By \eqref{GapN} and \eqref{naturally}, we get
 $$
  |\xi_{2}| \geq \frac{\sqrt{3}}{2 \pi h\, \left|A_{4}\right|^{1/4}}12^{3/4} I^{3/4} |A_{4}|^{3/8}.
   $$
   Since $h = \frac{ 3 I^{1/2 -\epsilon}}{2}\geq
   \frac{ 3 I^{1/2 -\epsilon}}{\pi}$, 
   we obtain
   $$
  |\xi_{2}| \geq \frac{\sqrt{3} \pi}{6 \pi  I^{1/2 -\epsilon}\, \left|A_{4}\right|^{1/4}}12^{3/4} I^{3/4 } |A_{4}|^{3/8} = 
 12^{1/4} I^{1/4+\epsilon} |A_{4}|^{1/8}.
$$
Repeating this, we get
$$
  |\xi_{k}| \geq  12^{1/4} I^{1/4+\left(\frac{3^{k-1} - 1}{2}\right)\epsilon} |A_{4}|^{1/8}.
$$
In order to have
$$
  |\xi_{k}| < \sqrt{2} h^{3/4} \left|3I A_{4}\right|^{1/8},
$$
the integer $k$ must satisfy 
$$
12^{1/4} I^{1/4+\left(\frac{3^{k-1} - 1}{2}\right)\epsilon} |A_{4}|^{1/8} < \sqrt{2} h^{3/4} \left|3I A_{4}\right|^{1/8}.
$$
Substituting $h = \frac{ \sqrt{3} I^{1/2 - \epsilon }}{2}$ in the above inequality, we get
$$
\left(\frac{8}{3}\right)^{1/4} I^{\left(\frac{3^{k-1} - 1}{2} + \frac{3}{4}\right) \epsilon }  < I^{1/4}.
$$
Therefore, since $\left(\frac{8}{3}\right)^{1/4} > 1$, we obtain
$$
k-1 < \frac{\log \left( \frac{1}{2 \epsilon} - \frac{3}{2}  + 1\right)}{\log 3}.
$$
We conclude that the number of solutions $(x , y)$ related to a fixed fourth root of unity with $$
\left| \xi(x , y) \right|^4  < 4h^3\sqrt{\left|3I A_{4}\right|},
$$
 does not exceed 
$$
\left[ \frac{\log \left( \frac{1}{2 \epsilon} - \frac{1}{2} \right)}{\log 3} \right] + 1.
$$
This, together with Lemma \ref{largexi}, completes the proof of Theorem \ref{mainq<}.

\section{The elliptic curve $Y^2 = X^3 - N X$}\label{N+}

let $N$ be a positive square-free integer. An integer solution to the equation 
\begin{equation}\label{theell}
Y^2 = X^3 - N X = X (X^2 - N)
\end{equation}
gives rise to a positive integer solution $(x , y)$ to the equation $x^2 - D y^4 = \frac{-N}{D}$, by taking
$$
X = Dy^2, \, \, \textrm{and}\, \, X^2 - N = Dx^2.
$$
In the above change of variables $D$ is the square-free part of $X$  and $D \mid N$.  
From now on, we will focus on the
quartic equation
\begin{equation}\label{theeq}
X^2 - D Y^4 = k,
\end{equation}
where  $D >1$ is a square-free integer, $N$ is a positive  integer, and $k$  is a negative integer.  Since we assumed that $N$ is square-free, the integer $k$ is also square-free and is relatively prime to $D$. We conclude that the summation
 \begin{equation}\label{UDN}
 \sum_{D\mid N} U_{D},
 \end{equation}
 wherein  $U_{D}$ is an upper bound for the number of solutions to \eqref{theeq}, will provide an upper bound for the number of integral solutions to \eqref{theell}.

 Assume that $(x_{0},y_{0}) \in \mathbb{Z}^2$ with $x_{0} y_{0} \neq 0$ is a point on the curve
 \begin{equation}\label{mainPell}
\mathfrak{X}^2 - D \mathfrak{Y}^2 = k.
\end{equation}
 Let
$$
\alpha = x_{0} + y_{0} \sqrt{D},
$$
and for $i \in \mathbb{Z}$, define $x_{i}, y_{i} \in \mathbb{Z}$ as follows:
\begin{equation}\label{Pelli}
x_{i} + y_{i} \sqrt{D} = \alpha \epsilon_{D}^{i},
\end{equation}
where $\epsilon_{D}$ denotes the minimal unit of $\mathbb{Z}[\sqrt{D}]$ and is defined in the Introduction.
Then each pair $(x_{i} , y_{i}) \in \mathbb{Z}^2$ is a solution to \eqref{mainPell}. We refer to the set of all such solutions as the \emph{class of solutions} to \eqref{mainPell} associated to $(x_{0}, y_{0})$. Assume that  $(x , y) \in \mathbb{Z}^2$ is a solution of  \eqref{mainPell}. Sometimes for simplicity, we call $x + \sqrt{D}y$ a solution of \eqref{mainPell}.
It is known \cite{Nag} that  two points $(u_{1},  v_{1})$ and $(u_{2}, v_{2})$ on the curve \eqref{mainPell} are  associated if and only if  the numbers 
$$
\frac{u_{1} u_{2} - v_{1} v_{2} D}{k} \, \, \qquad \textrm{and}  \, \, \qquad \frac{v_{1} u_{2} - v_{2} u_{1} }{k}
$$
are both  integers.

Among all solutions $x + y \sqrt{D}$ of  \eqref{mainPell} belonging to a given class, say $C$, we choose a solution $x^{*} + y^{*} \sqrt{D}$ in the following way. Let $y^{*}$  be the least positive value of $y$ which occurs in $C$ and let $x^{*}$ be a positive integer satisfying  $x^{*^{2}} - D y^{*^{2}} = k$. Then by the way  $y^{*}$ was chosen, at least one of $x^{*} + y^{*} \sqrt{D}$ and $-x^{*} + y^{*} \sqrt{D}$ belongs to $C$ and is called the \emph{fundamental solution} of class $C$.
Nagell in \cite{Nag} showed the following.
\begin{lemma}\label{Nage}
Let $\epsilon_{D} = T + U\sqrt{D} > 1$ be the minimal unit in the ring $\mathbb{Z}[\sqrt{D}]$ and $x^{*} + y^{*} \sqrt{D}$ be the fundamental solution of the equation $x^2 - Dy^2 = k$ in a given class. 
We have 
$$
0 < y^{*} < \frac{U}{\sqrt{2(T-1)}} \sqrt{|k|},
$$
and
$$
0 < |x^{*}| < \sqrt{\frac{(T-1)|k|}{2}}.
$$
\end{lemma}

 For a nonzero integer $k$, $\omega(k)$ denotes the number of distinct prime factors of $k$.  Let $n(D, k)$ denote the number of classes of coprime solutions $(x , y)$ to the quadratic equation $x^2 - Dy^2 = k$.  Walsh in \cite{Wal} showed (see Corollary 3.1 of \cite{Wal}) that if $k$ is square-free then
\begin{equation}\label{WalshCor3.1}
n(D, k) \leq 2^{\omega(k)}.
\end{equation}

\section{Reduction to  quartic Thue Equations}\label{reduction}

The idea of reducing the problem of counting integral points on elliptic curves to the problem of counting integral points of a finite number of quartic Thue  equations is not new. Chapter 27 of \cite{Mor} gives a complete overview of such reduction. 

Throughout this section, $D$ is a fixed positive square-free integer and $k$ is  a fixed negative  integer. Let $(X , Y) \in \mathbb{Z}^2$  satisfy the equation $X^2 - DY^4 = k$, then
 $X + Y^2 \sqrt{D}$ is a solution to equation \eqref{mainPell} belonging to a certain class $C$ of solutions. Let $x^{*} + y^{*} \sqrt{D}$ be the fundamental solution of $C$. Then 
\begin{equation}\label{Class}
 X + Y^2 \sqrt{D} = \left(x^{*} + y^{*} \sqrt{D}\right) \epsilon_{D}^{i}.
\end{equation}
Set
\begin{eqnarray*}
s_{2} + t_{2} \sqrt{D} = \left(x^{*} + y^{*} \sqrt{D}\right) \, \qquad \textrm{if $i$ is even},\\
s_{1} + t_{1} \sqrt{D} = \left(x^{*} + y^{*} \sqrt{D}\right) \epsilon_{D} \, \qquad \textrm{if $i$ is odd}.
\end{eqnarray*}
Therefore, there exists a non-negative integer $j$, such that
\begin{equation}\label{defoft}
 X + Y^2 \sqrt{D} = \left(s + t \sqrt{D}\right) \epsilon_{D}^{2j},
 \end{equation}
 where either $s + t \sqrt{D} = s_{1} + t_{1} \sqrt{D}$ or $s + t \sqrt{D} = s_{2} + t_2 \sqrt{D}$. Let
 $$
 m + n \sqrt{D} = \epsilon_{D}^{j},
 $$
 where $\epsilon_{D} =T + U\sqrt{D}$, with $T , U > 0$.
 Then we have
 $$
 m^2 - D n^2 = 1
 $$
 and 
 $$
 Y^2 = t m^2 + 2 s m n + t D n^2.
 $$
 Multiplying the above identity  by $t$, completing the square, and using the fact that $s^2 - D t^2 = k$, we obtain
 \begin{equation}\label{csq}
 -(tm + sn)^2 + k n^2 + t Y^2 = 0.
 \end{equation}

 Walsh  in \cite{Wal} showed the following.
\begin{lemma}\label{L2.1Walsh}
Let $a$, $b$, $c$ be nonzero integers with $\gcd(a, b, c) =1$, and such that the equation
\begin{equation}\label{W2.1}
ax^2 + b y^2 + c z^2 = 0
\end{equation}
has a solution in integers $x$, $y$ and $z$ not all zero. Then there are integers $R_{1}$, $S_{1}$, $T_{1}$, $R_{2}$, $S_{2}$, $T_{2}$, $z_{1}$, depnding only on $a$, $b$ and $c$, satisfying the relations
\begin{equation*}
R_{1} T_{2}  + R_{2} T_{1} = 2S_{1} S_{2},
\end{equation*}
\begin{equation*}
S_{2}^{2} - R_{2} T_{2} = -a c z_{1}^2,
\end{equation*}
\begin{equation*}
S_{1}^{2} - R_{1} T_{1} = -b c z_{1}^2
\end{equation*}
and a nonzero integer $\delta$, depending only on $a$, $b$, $c$, such that for every nonzero solution $(x , y , z)$ of \eqref{W2.1}, there exist integers $Q$, $u$, $v$ and a divisor $P$ of $\delta$, so that
\begin{eqnarray*}
P x &=& Q (R_{1} u^2 - 2S_{1} uv + T_{1} v^2)\, \, \textrm{and}\\
P y & = & Q (R_{2} u^2 - 2S_{2} uv + T_{2} v^2).
\end{eqnarray*}
The integers $R_{1}$, $R_{2}$, $T_{1}$, $T_{2}$ satisfy $R_{1} T_{2}  - R_{2} T_{1} = 0$.
\end{lemma} 

Applying Lemma  \ref{L2.1Walsh}
 to \eqref{csq}, with $a = -1$, $b = k$ and $c = t$, we conclude that 
 there are integers $R_{1}$, $S_{1}$, $T_{1}$, $R_{2}$, $S_{2}$, $T_{2}$ and $z_{1}$, depending only on $t$ and $k$, satisfying the relations 
\begin{equation}\label{W2.3a}
S_{2}^{2} - R_{2} T_{2} =  t z_{1}^2,
\end{equation}
\begin{equation}\label{W2.3b}
S_{1}^{2} - R_{1} T_{1} = -k  t z_{1}^2,
\end{equation}
\begin{equation}\label{W2.3c}
R_{1}T_{2} + R_{2} T_{1} = 2S_{1} S_{2}.
\end{equation}
and
\begin{equation}\label{W2.3d}
R_{1} T_{2}  - R_{2} T_{1} = 0.
\end{equation}
Also
\begin{eqnarray} \label{Pm}
P(t m + s n ) &=& Q (R_{1} u^2 - 2S_{1} uv + T_{1} v^2), \\ \label{Pn}
P n & = & Q (R_{2} u^2 - 2S_{2} uv + T_{2} v^2).
\end{eqnarray}
Therefore,
$$
\frac{t m + s n } {n} = \frac{  R_{1} \left(\frac{u}{v}\right)^2 - 2S_{1} \left(\frac{u}{v}\right)  + T_{1}   }{ R_{2}\left(\frac{u}{v}\right) ^2 - 2S_{2}\left(\frac{u}{v}\right)  + T_{2}}
$$
The above identity allows us to compute $\frac{u}{v}$ in terms of $m$ and $n$, by solving a quadratic equation.  We obtain $\frac{u}{v}$ is equal to one of the following two possible values:
\begin{eqnarray*} \nonumber
&  &\frac{  S_{1} -S_{2} l \pm \sqrt{ \left(S_{1} -S_{2} l \right)^2 - \left(R_{1} - R_{2}l\right)  \left(T_{1} - T_{2} l \right)  } }{R_{1} - R_{2}l },
\end{eqnarray*}
where $l = t \frac{m}{n} + s$.
The quantity under the square root in the above line can be simplified. Using \eqref{W2.3a},  \eqref{W2.3b}, \eqref{W2.3c} and \eqref{W2.3d}, we have
\begin{equation}\label{mnthingy}
\frac{u}{v}= \frac{  S_{1} -S_{2} \left(t \frac{m}{n} + s\right) \pm \sqrt{2 t^2z_{1}^2 \sqrt{D} (t \frac{m}{n} + s) }}{R_{1} - R_{2} \left(t \frac{m}{n} + s\right)}.
\end{equation}

Now solving \eqref{Pm} and \eqref{Pn}  for $m$ and $n$ and using the equation $m^2 - D n^2 =1$, it follows that
\begin{equation}\label{W4.2}
A^{2}_{1}(u , v)  - D  A^{2}_{2}(u , v)= \left(P t / Q\right)^2,
\end{equation}
where
$$
A_{1}(u, v) = \left(R_{1} - sR_{2}\right) u^2 - 2 (S_{1} - s S_{2}) u v + \left(T_{1} - s T_{2}\right) v^2
$$
and
$$
A_{2}(u, v) =R_{2} t u^2 - 2 S_{2} t u v + T_{2} t v^2.
$$

 We define
\begin{eqnarray}\label{defofFuv}
F(u , v) & =  & a_{0} u^4 + a_{1} u^3 v + a_{2} u^2 v^2 + a_{3} u v^3 + a_{4} v^4 \\ \nonumber
&:  = & A^{2}_{1}(u , v)  - D  A^{2}_{2}(u , v)= \left(P t / Q\right)^2.
\end{eqnarray}
Therefore,
\begin{eqnarray}\label{coeff}
a_{0} & = & R_{1}^2 - 2s R_{1} R_{2} + k R_{2}^2,\\ \nonumber
a_{1} & = & -4 (R_{1} S_{1} - s R_{1} S_{2} - sR_{2}S_{1} + k R_{2} S_{2}),\\ \nonumber
a_{2} & = & 6 (R_{1}T_{1} - sR_{2}T_{1} -s R_{1}T_{2} + kR_{2}T_{2}),\\  \nonumber
a_{3} & = & -4 (S_{1} T_{1} - sS_{1}T_{2} - sS_{2} T_{1} + k S_{2} T_{2}),\\  \nonumber
a_{4} & = & T_{1}^2 - 2s T_{1} T_{2} + k T_{2}^2.
\end{eqnarray}

Notice that since $m$ and $n$ are relatively prime, from \eqref{Pm} and \eqref{Pn}, we have 
$\gcd(u , v) =1$ and we are only interested in the primitive solutions of the Thue equations in \eqref{defofFuv}.

  We claim that the roots of polynomial $F(Z , 1)$ are real. The roots of the polynomial $F(z , 1)$ are
$$
\frac{  S_{1} -S_{2} o \pm \sqrt{ \left(S_{1} -S_{2} o \right)^2- \left(R_{1} - R_{2}o \right)  \left(T_{1} - T_{2}o \right)  } }{R_{1} - R_{2} o }.
$$
where $o =  s + t\sqrt{D}$ or $o= s- t\sqrt{D} $.
After routine simplification and using identities  \eqref{W2.3a},  \eqref{W2.3b},  \eqref{W2.3c} and \eqref{W2.3d}, we have
\begin{eqnarray*}
& & 4 \left(S_{1} -S_{2} (s \pm t\sqrt{D})\right)^2 \\
&-& 4 \left(R_{1} - R_{2}( s \pm t\sqrt{D})\right) \left(T_{1} - T_{2}( s \pm t\sqrt{D})\right)\\
& = & 2 t^2z_{1}^2 \sqrt{D} (\pm s + t \sqrt{D}).
\end{eqnarray*}
Since $s^2 - D t^2 = k < 0$, both $s+ t \sqrt{D}$ and $-s+ t \sqrt{D}$ are positive. Therefore, $F(z , 1)$ has $4$ real roots 
\begin{equation}\label{4Rroots}
\frac{  S_{1} -S_{2} (s \pm t\sqrt{D}) \pm \sqrt{  2 t^2z_{1}^2 \sqrt{D} (\pm s + t \sqrt{D})  } }{R_{1} - R_{2} (s \pm t\sqrt{D}) }.
\end{equation}
 
Let  $(u , v) \in \mathbb{Z}^2$ be a solution to our Thue equation that arises from a solution to the equation $X^4 - D y^2 = k$. By  \eqref{mnthingy}
 and since  $m$ and $n$ are positive integers with $\left(\frac{m}{n}\right)^2 = D + \frac{1}{n^2}$, we conclude that in order to give a bound upon the number of solutions to $X^4 - D y^2 = k$, we only need to count the number of solutions to the Thue equation $F(u , v) = \left(\frac{Pt}{Q}\right)^2$ that are associated to the following two real  roots:
 $$
\frac{  S_{1} -S_{2} (s + t\sqrt{D}) \pm \sqrt{  2 t^2z_{1}^2 \sqrt{D} (s + t \sqrt{D})  } }{R_{1} - R_{2} (s+ t\sqrt{D}) }.
$$

\section{ Quartic Thue equations, proof of Theorems  \ref{main=} and \ref{missproof}}

In Section \ref{reduction}, we constructed at most $2$ Thue equations for each fundamental solution of the equation $X^2 - DY^4 = k$ (see  \eqref{defoft} and 
\eqref{Class}).
Let $F(u , v)$ be the quartic binary form with coefficients given in \eqref{coeff}.  We showed that  $F(x , y)$ splits completely over $\mathbb{R}$ (see \eqref{4Rroots}). It turns out that $J_{F} = 0$.

\begin{lemma}\label{IJ}
Let $F(u , v)$ be the quartic binary form with coefficients given in \eqref{coeff}. Then we conclude that 
$$ J_{F} = 0 \, \, \textrm{and} \, \, I_{F} =  48 k t^3 T_{2} R_{2} z_{1}^2 D. 
$$
\end{lemma}
\begin{proof}
Walsh in \cite{Wal} showed that $J_{F} = 0$. Here we will compute the value for invariant $I$.
Recall that
$$ I_F = 12 a_{0} a_{4} - 3 a_{1} a_{3} + a_{2}^2.
$$
First we compute the value of each summand in the above summation, using \eqref{coeff}
and relations \eqref{W2.3a}, \eqref{W2.3b} and \eqref{W2.3c}.
\begin{eqnarray*}
12 a_{0} a_{4} & = & (R_{1}^2 - 2s R_{1} R_{2} + k R_{2}^2) (T_{1}^2 - 2s T_{1} T_{2} + k T_{2}^2)\\
& =& 12 (R_{1}^2 T_{1}^2 - 2 s R_{1}^2 T_{1} T_{2} + kR_{1}^2 T_{2}^2\\
&& - 2s R_{1} R_{2} T_{1}^2 + 4 s^2 R_{1} R_{2} T_{1} T_{2} - 2 s k R_{1} R_{2} T_{2}^2\\
&&  + kR_{2}^2 T_{1}^2 - 2 k s R_{2}^2 T_{1} T_{2} + k^2 R_{2}^2 T_{2}^2),
\end{eqnarray*}
\begin{eqnarray*}
& & - 3 a_{1} a_{3} = \\
 & & -  48 (R_{1} S_{1} - s R_{1} S_{2} - sR_{2}S_{1} + k R_{2} S_{2}) \times\\
 &\times& (S_{1} T_{1} - sS_{1}T_{2} - sS_{2} T_{1} + k S_{2} T_{2}) \\
& = & -48\left( R_{1}S_{1}^2 T_{1} - sR_{1} S_{1}^2 T_{2} - s R_{1} S_{1} S_{2} T_{1} + k R_{1} S_{1}S_{2} T_{2} \right.\\
& & - sR_{1} S_{2} S_{1}T_{1} +s^2 R_{1} S_{2} S_{1}T_{2} + s^2 R_{1} S_{2}^2 T_{1} -sk R_{1} S_{2}^2 T_{2}\\
&&- sR_{2}  S_{1}^{2} T_{1} +s^2 R_{2}  S_{1}^{2} T_{2} + s^2 R_{2} S_{1}S_{2} T_{1} -sk R_{2}S_{1} S_{2} T_{2}\\
& & + \left.  kR_{2} S_{2} S_{1} T_{1} - k s R_{2} S_{2} S_{1} T_{2} -ks R_{2} S_{2}^2 T_{1} + k^2 R_{2} S_{2}^2 T_{2}\right),
\end{eqnarray*}
and
\begin{eqnarray*}
a_{2}^2 & = & 36 (R_{1}T_{1} - sR_{2}T_{1} -s R_{1}T_{2} + kR_{2}T_{2})^2  \\
& = & 36 (R_{1}^{2} T_{1}^{2}+ s^{2} R_{2}^{2} T_{1}^{2} +s^{2} R_{1}^{2}T_{2}^{2} + k^{2} R_{2}^{2}T_{2}^{2}\\
&& - 2s R_{1} R_{2} T_{1}^2 -2 s  R_{1}^2 T_{1} T_{2} + 2k R_{1} R_{2} T_{1} T_{2}\\
& & + 2 s^{2}R_{1} R_{2} T_{1} T_{2} -2sk R_{2}^2 T_{1} T_{2} -2k R_{1} T_{2}^{2} R_{2}).
\end{eqnarray*}

Therefore, 
\begin{eqnarray*}
I_F & = & 96 ks R_{1} T_{2} (S_{2}^{2} - T_{2} R_{2}) + 96 s R_{2} T_{1} (S_{1}^{2} - R_{1} T_{1})\\
& & -48 k^2 R_{2} T_{2} (S_{2}^{2} - R_{2} T_{2}) -48 s^2 R_{2} T_{2} (S_{1}^2 - R_{1} T_{1})\\
& = & 48 k T_{2} (S_{2}^{2} - R_{2} T_{2})(2s R_{1} - k R_{2})\\
&&  + 48 s R_{2} (S_{1}^{2} - R_{1} T_{1})(2T_{1} - s T_{2})\\
& = & 48 kt z_{1}^2 T_{2} R_{2} (s^2 - k)\\
& = & 48 k t^3 T_{2} R_{2} z_{1}^2 D,
\end{eqnarray*}
where the last identity is because $s^2 - k =  D t^2$.
\end{proof}

We recall some results for the number of solutions of the quartic Thue equations.

\begin{prop}\label{BSCH}
Let $\mathfrak{S}$ be the set of quartic  forms $F(x , y) \in \mathbb{Z}[x , y]$ that are irreducible over $\mathbb{Q}$ with $I_{F} > 0$ and $J_{F} = 0$.
Let $\mathfrak{N}$ be an upper bound for the number of solutions of quartic Thue equations
$$
F(x , y) = 1
$$
as $F$ varies over the elements of $\mathfrak{S}$. Then for $h \in \mathbb{N}$ and $G(x , y) \in \mathfrak{S}$, the equation 
$$G(x , y) = h$$
has at most $$\mathfrak{N} \, 4^{\omega(h)}$$ primitive solutions.
\end{prop}
\begin{proof}
This is essentially a special case of Bombieri and Schmidt's result in \cite{Bom}, where they showed that if $N_{n}$ is an upper bound for the number of solutions to the equations $F(x, y) = 1$, as $F(x , y)$ varies over irreducible binary forms of degree  $n$ with integer coefficients  then $N_{n}n^{\omega(h)}$ is an upper bound  for the number of primitive solutions to $F(x, y) = h$. Bombieri and Schmidt proved this fact by reducing a given Thue equation $F(x , y) = h$ modulo every prime factor of the integer $h$.  This reduction is explained in the proof of Lemma 7 of \cite{Bom}, where the form $F(x , y)$ of degree $n$ is reduced to some other binary forms  of degree $n$. These reduced forms  are basically obtained through the action of $2 \times 2$   matrices with integer arrays and non-zero discriminant  on the binary form $F(x , y)$.  Since $J$ is an invariant, a quartic form  $G(x , y)$ with $J_F = 0$ will be reduced to some quartic forms with $J =0$ under the action of $2 \times 2$ matrices. Also since $I_{F} > 0$ is an invariant of weight $4$ (an even number), we will get forms with positive $I$ under the action of $2 \times 2$  matrices with non-zero discriminant.
\end{proof}

   \begin{prop}\label{main2ofThue-Siegel}
 Let $F(x , y) \in \mathbb{Z}[x , y]$ be a quartic form with positive discriminant that is irreducible over $\mathbb{Q}$ and splits in $\mathbb{R}$. If $J_{F} = 0$, then the Diophantine equation $\left|F(x , y)\right| = 1$ possesses at most $12$ solutions in integers $x$ and $y$ (with $(x , y)$ and $(-x , -y)$ regarded as the same).
 \end{prop}
 \begin{proof}
This is Theorem 1.1 of  \cite{Akh4}.
 \end{proof}
 
 From Propositions \ref{BSCH}  and \ref{main2ofThue-Siegel}, we conclude the following.
 \begin{cor}\label{Q4h}
  Let $F(x , y) \in \mathbb{Z}[x , y]$ be a quartic form with positive discriminant that is irreducible over $\mathbb{Q}$ and splits in $\mathbb{R}$.  If $J_{F} = 0$ and $h$ is a positive integer, then the Diophantine equation $\left|F(x , y)\right| = h$ possesses at most $12. 4^{\omega(h)}$ primitive solutions  (with $(x , y)$ and $(-x , -y)$ regarded as the same).
 \end{cor}

Lemma \ref{IJ} and Corollary \ref{Q4h} imply that  the Thue equation $F(u , v) = \frac{P^2t^2}{Q^2}$, with coefficients given in \eqref{coeff}, has at most $12 . 4^{\omega(\frac{P^2t^2}{Q^2})}$ primitive solutions. From the proof of Lemma 2.1 of \cite{Wal} and equation (6.5) and the two unnumbered equations above  (6.5) in \cite{Wal}, we have
$$
|P| = |C|,
$$
and $c_{0} C^2 = t$, where $C$ and $c_{0}$ are integers.
Therefore, $P | t$ and $\omega(\frac{P^2t^2}{Q^2}) \leq \omega(t)$.
By  Lemma \ref{Nage} and the  definition of $t$ in \eqref{defoft}, we have
\begin{equation}\label{tless}
t \leq \epsilon_{D}^{3/2} \sqrt{|k|/2D}.
\end{equation}
We have 
$$\omega(\frac{P^2t^2}{Q^2}) \leq \omega(t) < 2 + \log t/\log 4 \leq 2 + \frac{\log \left( \epsilon_{D}^{3/2} \sqrt{|k|/ 2D} \right)  }{\log 4}.
$$
Therefore, the Thue equation $F(u , v) = \frac{P^2t^2}{Q^2}$, with coefficients given in \eqref{coeff}, has at most $12 . 16.   \epsilon_{D}^{3/2} \sqrt{|k|/ 2D} $ primitive solutions.
This completes the proof of Theorem \ref{main=}, as 
we have $2 \, n(D , k)$ Thue equations $F(u , v) = \frac{P^2 t^2}{Q^2}$ (see  \eqref{defofFuv}) associated to equation  \eqref{mainqequ} and by \eqref{WalshCor3.1}, $n(D , k) \leq 2^{\omega(k)}$.

By \eqref{UDN}, and since $|k| = N/D$, the following is an upper bound for the number of integer solutions of the equation \eqref{maincequ}:
 \begin{eqnarray*}
& & \sum_{D| N} 384 \, .  2^{\omega(N/D)}   \epsilon_{D}^{3/2} \sqrt{N/ 2D^2}\\ \nonumber & =& 384 \sqrt{N/2}  \sum_{D| N}  \frac{2^{\omega(N/D)}   \epsilon_{D}^{3/2}  }{D}.
\end{eqnarray*}
This completes the proof of Theorem \ref{missproof}.


\section{ Proof of Theorem \ref{mainqe}} \label{compute}

In Section \ref{reduction}, we constructed at most $2$ Thue equations for each fundamental solution of the equation $X^2 - DY^4 = k$ (see  \eqref{defoft} and 
\eqref{Class}).

Let $F(u , v)$ be the quartic binary form with coefficients given in \eqref{coeff}. In Lemma \ref{IJ}, we  showed that $J_{F} =0$ and $I_{F} =  48 k t^3 T_{2} R_{2} z_{1}^2 D$.
 Therefore we may apply Theorem \ref{mainq<} to the inequality $|F(u , v)| \leq \frac{P^2 t^2}{Q^2}$ by taking   
 $\epsilon = \frac{1}{12}$.
 Then by Lemma \ref{IJ}, we have
$$
I^{1/2 - \epsilon} =  I ^{5/12}=  48^{5/12} k^{5/6} t^{5/4} D^{5/12} P^{5/6} y_{1}^{5/6}  z_{1}^{5/6}.
$$
Now we observe that
\begin{equation}\label{observeh}
h = \frac{P^2 t^2}{Q^2} \leq P^{5/6} t^{4/3} P^{7/6} t^{2/3} \leq P^{5/6} t^{4/3} t^{7/12} t^{2/3},
\end{equation}
because $Q$ is an integer.
From  \eqref{tless},
$$
t \leq \epsilon_{D}^{3/2} \sqrt{|k|/2D}.
$$
We conclude that
$$
h \leq P^{5/6} t^{4/3} t^{5/4} \leq P^{5/6}  t^{5/4} \epsilon_{D}^{2} (|k|/2D)^{2/3}.
$$
In order to apply  Theorem \ref{mainq<} to the inequality $|F(u , v)| <  \frac{P^2 t^2}{Q^2}$, the following must hold. 
$$
h < \frac{\sqrt{3}\,   I^{1/2 }}{\pi}.
$$
We claim that if we assume
$$
|k| \geq \left(\frac{\pi}{\sqrt{3} \times 2^{2/3} 48^{5/12}}\right)^6 \frac{\epsilon_{D}^{12} }{D^{13/2}} =  \frac{\pi}{2^{14} 3^{11/2}} \frac{\epsilon_{D}^{12} }{D^{13/2}}
$$
then we have $
h \leq \frac{\sqrt{3}\,   I^{5/12 }}{\pi}$. This is because under this assumption, we have
$$
|k|^{1/6} \geq \epsilon_{D}^{2} (1/2D)^{2/3} \frac{\pi}{\sqrt{3}}\frac{1}{48^{5/12}D^{5/12}},
$$
which holds if 
$$
\epsilon_{D}^{2} (|k|/2D)^{2/3} \leq  \frac{ \sqrt{3} }{\pi} 48^{5/12} k^{5/6}  D^{5/12}.
$$
This, together with \eqref{observeh}, implies that 
$$
h  \leq C^{5/6} t^{4/3} t^{7/12} t^{2/3} \leq  \frac{\sqrt{3}\,   I^{5/12 }}{\pi}.
$$

In \eqref{defofFuv},
we got $2 \, n(D , k)$ Thue equations $F(u , v) = \frac{P^2 t^2}{Q^2}$. 
The primitive solutions to these equations form a subset of the primitive solutions to the Thue inequalities $F(u , v) \leq \frac{P^2 t^2}{Q^2}$.
Taking $\epsilon = \frac{1}{12}$ in Theorem \ref{mainq<}, each of these inequalities has at most $20$ primitive solutions. Therefore, we get at most   $40 \, n(D , k)$ primitive solutions.  This, together with \eqref{WalshCor3.1},
completes the proof of  Theorem \ref{mainqe}.

\section{The elliptic curve $Y^2 = X^3 +N X$}\label{+N}


Let $N$ be a square-free positive integer.
 An integer solution to the equation 
$$
Y^2 = X^3 +N X
$$
gives rise to a positive integer solution $(x , y)$ to the equation $x^2 - D y^4 = \frac{N}{D}$, by taking
$$
X = Dy^2, \, \, \textrm{and}\, \, X^2 +N = Dx^2.
$$
In the above change of variables $D$ is the square-free part of $X$  and $D \mid N$.
Tzanakis in \cite{Tza} showed the following.
\begin{thm}\label{Ttza}
Let $D$ and $k$ be positive  integers which are not perfect squares. Then all integers solutions to the equation
$$
X^2 - Dy^4 = k
$$
can be found by finding the integral solutions to a finite number of quartic Thue equations of the form
$$
g(u , v) = B^2,
$$
where the polynomial $g(Z, 1)$ has exactly two real roots.
\end{thm}

We  remark that in Theorem \ref{Ttza}, the fact that  $g(Z , 1)$ has exactly  two real roots for every form $g$ plays a very important role. Let $g$ be irreducible. Then we may write 
$$
\textrm{Norm} (u - v \alpha) = B^2,
$$
which is equivalent to a finite number of equations
\begin{equation}\label{TzaE1.9}
u - v \alpha = \beta u_{1}^{m} u_{2}^{n},
\end{equation}
where $\beta$ runs through a finite set of algebraic integers in $\mathbb{Q}(\alpha)$ with
 $$\textrm{Norm}(\beta) = B^2,$$
  and $u_{1}$ and $u_{2}$ is a pair of fundamental units in some order of  $\mathbb{Q}(\alpha)$. Equation \eqref{TzaE1.9} is an exponential equation in unknowns $m$ and $n$  and there are two equations relating them, which are obtained by equating the coefficients of 
 $\alpha^2$ and $\alpha^3$ to $0$ in $u_{1}^{m} u_{2}^{n}$. The $p$-adic method can be applied in this situation (see, for example, \cite{Tza},
 \cite{Mor} and \cite{Lew0}).
One may attempt to diagonalize the binary form $g(u , v)$, so that
$$
g(u , v) = \xi^4(u , v) - \eta^4(u , v).
$$ 
As opposed to the case in which the binary form splits completely in $\mathbb{R}$, in this case, where $g(z, 1)$ has $2$ real roots and $2$ non-real roots, both linear forms $\xi(u , v)$ and $\eta(u , v)$ have real coefficients and $\xi(u , v)$ and $\eta(u , v)$ are not complex conjugates. Therefore, results such as Lemma \ref{6.12} and the concomitant gap principles  will not work in this case and the method that we applied to count the number of integer points on $Y^2 = X^3 -N X$ cannot be used for the  curve $Y^2 = X^3 + N X$.




\end{document}